\documentclass{amsart}
\usepackage{amsmath}
\usepackage{txfonts}
\usepackage{amsfonts}
\usepackage{hyperref}

\newtheorem{theorem}{Theorem}[section]
\newtheorem{lemma}[theorem]{Lemma}
\newtheorem{proposition}[theorem]{Proposition}

\theoremstyle{definition}
\newtheorem{definition}[theorem]{Definition}

\theoremstyle{remark}
\newtheorem{remark}[theorem]{Remark}

\numberwithin{equation}{section}

\begin{document}

\title{curvature and integrability of almost complex structures}

\author{}
\address{School of Mathematics, Northwest University, Xi'an 710127, China}
\curraddr{} \email{wanj\_m@aliyun.com}
\thanks{The author is supported by National Natural Science Foundation of China
No.11301416 and No.11601421.}

\author{Jianming Wan}
\address{}
\email{}
\thanks{}

\subjclass[2010]{Primary 53C15; Secondary 53C20}

\date{}

\dedicatory{}

\keywords{curvature, almost complex structure}

\begin{abstract}
This note is concerned in so called harmonic complex structures introduced by the author previously. I will recall some previous results and emphasize the motivation: Provide an attempt to a fundamental problem in geometry--determining the complex structures on an almost complex manifold.  I also discuss the almost-Hermitian case of harmonic complex structures and the connections with balanced structures.
\end{abstract}
\maketitle

\section{Introduction}
Let $M$ be an almost complex manifold. This means that there exists a smooth section $J\in \Gamma(T^{*}M\bigotimes TM)=\Gamma(\hom(TM, TM))$ satisfying
\begin{equation}
J^{2}=-id.
\end{equation}
 The $J$ is called almost complex structure of $M$. An almost complex manifold must be oriented and has even dimension. Determining an almost
complex structure on a manifold is a purely topological problem (equivalent to the structure groups of tangent bundle $GL(2n,\mathbb{R})$ can reduce to $GL(n,\mathbb{C})$) and has been studied well \cite{[S]}. In principle, we always can determine whether a given manifold has an almost complex structure (though the procedure may be very complex). In dimension 4 we have a fundamentally topological criterion of Wen-tsun Wu to determining the almost complex structure \cite{[Wu]}.

An almost complex structure $J$ on $M$ is said to be integrable, if it can induce a complex (manifold) structure.
By the famous Newlander-Nirenberg theorem \cite{[NN]}, J is integrable if and only if the Nijenhuis tensor vanishes, i.e.
\begin{equation}
N(J)(X,Y)=J[JX,Y]+J[X,JY]+[X,Y]-[JX,JY]\equiv 0, \label{f1.2}
\end{equation}
 for all $X,Y \in \Gamma (TM)$.
 \emph{A fundamental problem in geometry is to determine the complex structures on an almost complex manifold.}
 The case of 2-dimension is classical. Every surface is a complex manifold, which is called Riemann surface. In dimension 4,
 by combining Wu's criterion with some results in algebraic geometry we can construct many compact almost complex manifolds without
 any complex structure (c.f. \cite{[BPV]} page 167). For instance, $S^{1}\times S^{3}\sharp S^{1}\times S^{3}\sharp S^{2}\times S^{2}$
 and $S^{1}\times S^{3} \sharp S^{1}\times S^{3}\sharp \mathbb{C}P^{2}$. However, up to now, we do not find a single higher dimensional
 manifold with almost complex structures but no complex structure. The higher dimensional examples seem to be existent undoubtedly.
 But another opinion of Yau (c.f. \cite{[Ya]} problem 52) asserts that every compact almost complex manifold of dimension $\geq 6$ admits a complex structure.
  As well known, one can construct an almost complex structure on $S^{6}$ by using quaternions. But this almost complex structure is not integrable.
   It is an outstanding problem to determine the complex structures on $S^{6}$. $S^{6}$ is a touchstone to understand the complex structures of higher
   dimensional manifolds .

To deal with the fundamental problem, there are two folds: 1) To show the existence of complex structures,
we should find some effective sufficient conditions to the existence of complex structures; 2) To show the nonexistence of complex structures,
 we need to find some obstructions. The sufficient conditions or obstructions should be involved in the geometry or topology of manifolds.

We are mainly concerned in the connections between complex structures and curvature of manifolds. Let $(M,J)$ be an almost complex manifold.
We give a Riemannian metric $g$ on $M$. Then we can define the Hodge-Laplace operator $\Delta$ acting on tangent bundle-valued differential forms.
 Since $J$ can be seen as a tangent bundle-valued 1-forms, we may consider the action of $\Delta$ on $J$. When the manifold is compact,
 in \cite{[W]} the author observed that
 \begin{equation} \Delta J=0 \end{equation} implies $J$ is integrable. On the other hand, the Bochner formula of $\Delta J$ contains curvature terms.
  So we can connect the integrability of almost complex structures with the geometry (curvature) of manifolds. In a sense this provides a probability for
   studying of existence of complex structures.

\section{Bochner techniques for tangent bundle-valued differential forms}
The materials in this section are standard, which can be found in many literatures. For example \cite{[X]}.
\subsection{Hodge-Laplace operator}
Notations.  $\{e_{i}, 1\leq i\leq n\}$: local orthonormal frame field; $X, X_{0}, X_{1},\cdot\cdot\cdot,X_{p}$: smooth sections of tangent bundle $TM$; $\omega,\theta$: smooth sections of $\wedge^{p}T^{*}M\otimes TM$.

Let $(M, g)$ be a Riemannian manifold. Let $\nabla$ be the Levi-Civita
connection associated with $g$. $\nabla$ can be extended canonically to
$\Gamma(\wedge^{p}T^{*}M\otimes TM)$ by
$$(\nabla_{X}\omega)(X_{1},\cdot\cdot\cdot, X_{p})=\nabla_{X}(\omega(X_{1},\cdot\cdot\cdot, X_{p}))-\sum_{k=1}^{p}
\omega(X_{1},\cdot\cdot\cdot, \nabla_{X}X_{k},\cdot\cdot\cdot, X_{p}).$$

We can define the differential operator $d :
\Gamma(\wedge^{p}T^{*}M\otimes TM) \longrightarrow
\Gamma(\wedge^{p+1}T^{*}M\otimes TM),$
\begin{equation}
d\omega(X_{0},\cdot\cdot\cdot,
X_{p})=\sum_{k=0}^{p}(-1)^{k}(\nabla_{X_{k}}\omega)(X_{0},\cdot\cdot\cdot, \hat{X}_{k},\cdot\cdot\cdot,
X_{p}),
\end{equation}
where $\hat{X}_{k}$ denotes removing ${X}_{k}$.
The co-differential operator $\delta : \Gamma(\wedge^{p}T^{*}M\otimes
TM) \longrightarrow \Gamma(\wedge^{p-1}T^{*}M\otimes TM)$ is given
by
\begin{equation}
\delta\omega(X_{1},\cdot\cdot\cdot, X_{p-1})=-\sum_{i=1}^{n}(\nabla_{e_{i}}\omega)(e_{i},
X_{1}, ..., X_{p-1}).
\end{equation}

The \textbf{Hodge-Laplace operator} is defined by
\begin{equation}
\Delta\triangleq d\delta+\delta d.
\end{equation}

For any $\omega,\theta \in
\Gamma(\wedge^{p}T^{*}M\otimes TM)$, we have the induced inner product
\begin{equation}
\langle\omega,\theta\rangle\triangleq \sum_{1\leq i_{1}<\cdot\cdot\cdot<i_{p}\leq n}\langle\omega(e_{i_{1}},\cdot\cdot\cdot,e_{i_{p}}), \theta(e_{i_{1}},\cdot\cdot\cdot,e_{i_{p}})\rangle,
\end{equation}
If $M$ is compact, we have the global inner product
\begin{equation}
(\omega, \theta) \triangleq \int_{M}\langle\omega, \theta\rangle dv.
\end{equation}
By the self-adjoint property of $\Delta$, we have
\begin{equation}
(\Delta \omega, \omega)= (d\omega, d\omega)+(\delta\omega, \delta\omega)\geq 0.
\end{equation}
So $\Delta\omega=0$ if and only if $d\omega=0$ and
$\delta\omega=0$.

We should mention that in general $d^{2}\neq 0$. For $A\in \Gamma(T^{*}M\bigotimes TM)$,
$d^{2}A(X_{1},X_{2},X_{3}) = R(X_{3},X_{2})AX_{1}+R(X_{1},X_{3})AX_{2}+R(X_{2},X_{1})AX_{3}$.

\subsection{Weitzenb\"{o}ck formula}
\begin{proposition} \label{p2.1}
 For any tangent bundle-valued $p$-form $\omega$, we have
\begin{equation} \Delta\omega=-\nabla^{2}\omega+S, \label{f2.7} \end{equation} where
$\nabla^{2}\omega=\nabla_{e_{i}}\nabla_{e_{i}}\omega-\nabla_{\nabla_{e_{i}}e_{i}}\omega$
and $S(X_{1},\cdot\cdot\cdot, X_{p})=(-1)^{k}(R(e_{i}, X_{k})\omega)(e_{i},
X_{1},\cdot\cdot\cdot, \hat{X_{k}},\cdot\cdot\cdot, X_{p}),$ for any $X_{1},\cdot\cdot\cdot,
X_{p}\in\Gamma(TM)$. $R$ is the curvature tensor $R(X,
Y)=-\nabla_{X}\nabla_{Y}+\nabla_{Y}\nabla_{X}+\nabla_{[X, Y]}$ and
$\{e_{i}\}$ is the local orthonormal frame field.
\end{proposition}

In fact the Bochner technique can work on any Riemannian vector bundles. For our purpose we only focus on tangent bundles.

\section{Harmonic complex structures}
From now on we assume that $(M, J)$ is a compact almost complex manifold. We give a Riemannian metric $g$ on $M$. $J$ \textbf{does not need} to be compatible with $g$.
The author introduced the following concept in \cite{[W]}.

\begin{definition}
We say that $J$ is a harmonic complex structure if $\Delta J=0$.
\end{definition}

By the self-adjoint property, $\Delta J=0$ if and only if $dJ=0$ and $\delta J=0$. Recall that a K\"{a}hler structure means an almost complex structure $J$ compatible with $g$ and satisfying $\nabla J=0$. So a K\"{a}hler structure must be a harmonic complex structure. The following observation shows the meaning of harmonic complex structure.

\begin{proposition}\label{p3.2}
\cite{[W]} A harmonic complex structure is integrable.
\end{proposition}

\begin{proof}
We only need to show that the Nijenhuis tensor \ref{f1.2} vanishes. By direct computation, one has
\begin{eqnarray*}
dJ(X, Y)-dJ(JX, JY) &=&((\nabla_{X}J)Y-(\nabla_{Y}J)X)-((\nabla_{JX}J)JY-(\nabla_{JY}J)JX)\\
&=& [JX, Y]+[X, JY]+J[JX, JY]-J[X, Y]\\
&=&-JN(J)(X,Y).
\end{eqnarray*}
 Since $\Delta J=0$ implies $dJ=0$, hence $N(J)=0$.
\end{proof}

Denote $K$: K\"{a}hler structures; $H$: harmonic complex structures; $C$: complex structures. We have the inclusion relation $$K \subset H \subset C.$$

Let $e(J)=\frac{1}{2}|Je_{i}|^{2}$ denote the energy density of $J$.  Applying proposition \ref{p2.1} to $J$, we can obtain the Bochner type formula.
\begin{theorem}\label{t3.3}
\cite{[W]}
\begin{equation}
\Delta e(J)+\langle\Delta J,
J\rangle=|\nabla J|^{2}+\langle JR(e_{i},
e_{j})e_{i}, Je_{j}\rangle-\langle R(e_{i}, e_{j})Je_{i}, Je_{j}\rangle,\label{f3.1}
\end{equation}
where  $|\nabla J|^{2}=|(\nabla _{e_{i}}J)(e_{j})|^{2}$.
\end{theorem}

\begin{proof}
Following the notations in proposition \ref{p2.1}, we can check that
$$\langle S, J\rangle=\langle JR(e_{i}, e_{j})e_{i},Je_{j}\rangle-\langle R(e_{i}, e_{j})Je_{i}, Je_{j}\rangle$$ and
 $$\langle\nabla^{2}J, J\rangle=\Delta e(J)-|\nabla J|^{2}.$$ Then the theorem is straightforward from formula \ref{f2.7}.
\end{proof}

The below table gives the comparative relations.

\begin{table}[ht]
\caption{}\label{eqtable}
\renewcommand\arraystretch{1.5}
\noindent\[
\begin{array}{|c|c|c|}
\hline
Kahler\ structure & harmonic\ complex\ structure\\
\hline
totally\ geodesic\ map  & harmonic\ map\\
\hline
totally\ geodesic\ submanifold & minimal\ submanifold\\
\hline
\end{array}
\]
\end{table}

\section{Some applications of harmonic complex structures}
From theorem \ref{t3.3}, we have $$(dJ,dJ)+(\delta J,\delta J)=(\Delta J,J)=\int_{M}(|\nabla J|^{2}+\langle JR(e_{i},
e_{j})e_{i}, Je_{j}\rangle-\langle R(e_{i}, e_{j})Je_{i}, Je_{j}\rangle)dv\geq0.$$
So $J$ is a harmonic complex structure if and only if $$\int_{M}(|\nabla J|^{2}+\langle JR(e_{i},
e_{j})e_{i}, Je_{j}\rangle-\langle R(e_{i}, e_{j})Je_{i}, Je_{j}\rangle)dv=0.$$

 Combining proposition \ref{p3.2}, we obtain a geometric sufficient condition of integrability of an almost complex structure.
\begin{theorem}\label{t4.1}
If \begin{equation} \int_{M}(|\nabla J|^{2}+\langle JR(e_{i},
e_{j})e_{i}, Je_{j}\rangle-\langle R(e_{i}, e_{j})Je_{i}, Je_{j}\rangle)dv=0, \label{f4.1} \end{equation} then $J$ is integrable.
\end{theorem}

Though we do not know whether the 6-sphere $S^{6}$ has a complex structure, as an
application of theorem \ref{t3.3} we have
\begin{theorem}\cite{[W]}
 $S^{6}$ with standard metric (or with small perturbation) can not admit any harmonic complex
structure.
\end{theorem}

\begin{proof}
We only need to show that the right of formula \ref{f3.1} is positive. Under standard metric, the sectional curvature is equal to 1.
One can easily to show that $\langle JR(e_{i}, e_{j})e_{i}, Je_{j}\rangle=10e(J)\geq 30$ and $\langle R(e_{i}, e_{j})Je_{i}, Je_{j}\rangle=6$.
\end{proof}

A well-known result of LeBrun \cite{[L]} states that $S^{6}$ has no complex structure compatible with the standard metric.
Kefeng Liu and Xiaokui Yang \cite{[Y1]} \cite{[Y2]} proved that $S^{6}$ can not admit a complex structure compatible with a metric such that the sectional curvature lies in $(\frac{1}{4},1]$. In our result the compatible condition is removed. But geometric restriction is increased.

We also can get a K\"{a}hler criterion for a harmonic complex structure.

\begin{theorem}\cite{[W]}
Let $J$ be an Hermitian harmonic complex structure. Then the scale curvature $S\leq \langle R(e_{i}, e_{j})Je_{i}, Je_{j}\rangle$. The equal holds if and only if $J$ is a K\"{a}hler structure.
\end{theorem}

\begin{proof}
Since $J$ is an Hermitian harmonic complex structure, $e(J)=constant$. The left of formula \ref{f3.1} equals to zero. So $S- \langle R(e_{i}, e_{j})Je_{i}, Je_{j}\rangle=-|\nabla J|^{2}\leq 0$. The equal holds implies $\nabla J=0$. Namely $J$ is a K\"{a}hler structure.

\end{proof}

\section{Almost-Hermitian case}

\subsection{Almost-Hermitian manifold}
We know that $\Delta J=0$ if and only if both $dJ$ and $\delta J$ are equal to zero.
Since we mainly are concerned in the integrability of almost complex structures, only the $dJ=0$ is useful for our purpose.
We should remove the condition $\delta J=0$. When $J$ is compatible with the Riemannian metric, i.e. $M$ is an almost-Hermitian manifold, we can do it.
Our main observation is

\begin{proposition} \label{p5.1}
Let $M$ be an almost-Hermitian manifold. Then $dJ=0$ implies $\delta J=0$.  More precisely, for any $X\in \Gamma(TM)$ we have \\
i) $\langle JX, \delta J\rangle+\langle dJ(X,e_{i}),Je_{i}\rangle\equiv 0$,\\
ii) $\langle X, \delta J\rangle+\langle dJ(X,e_{i}),e_{i}\rangle\equiv 0$.
\end{proposition}

\begin{proof}
We choose the normal frame field (i.e.
$\nabla_{e_{i}}e_{j}|_{p}=0$ for a fixed point $p$). Then $dJ(X,e_{i})=\nabla_{X}Je_{i}-\nabla_{e_{i}}JX+J\nabla_{e_{i}}X$. Hence
\begin{eqnarray*}
\langle dJ(X,e_{i}),Je_{i}\rangle &=& \langle\nabla_{X}Je_{i}, Je_{i}\rangle-\langle\nabla_{e_{i}}JX, Je_{i}\rangle+\langle J\nabla_{e_{i}}X, Je_{i}\rangle\\
&=& -\langle\nabla_{e_{i}}JX, Je_{i}\rangle+\langle \nabla_{e_{i}}X, e_{i}\rangle\\
&=& \langle JX, \nabla_{e_{i}}Je_{i}\rangle-e_{i}\langle X,e_{i} \rangle+div X\\
&=& -\langle JX,\delta J\rangle-div X+div X=-\langle JX,\delta J\rangle
\end{eqnarray*}
and
\begin{eqnarray*}
\langle dJ(X,e_{i}),e_{i}\rangle &=& \langle\nabla_{X}Je_{i}, e_{i}\rangle-\langle\nabla_{e_{i}}JX, e_{i}\rangle+\langle J\nabla_{e_{i}}X, e_{i}\rangle\\
&=& -\langle\nabla_{e_{i}}JX, e_{i}\rangle-\langle \nabla_{e_{i}}X, Je_{i}\rangle\\
&=& -div JX+\langle X,\nabla_{e_{i}}Je_{i}\rangle-e_{i}\langle X, Je_{i}\rangle\\
&=& -div JX-\langle X, \delta J\rangle+e_{i}\langle JX, e_{i}\rangle=-\langle X, \delta J\rangle.
\end{eqnarray*}

\end{proof}

Now we explain why $dJ=0$ can imply $\delta J=0$. Let $J_{t}, -\epsilon<t<\epsilon$ be a family of almost complex structures.
$J_{0}$ is compatible with the Riemannian metric. Then the energy density $e(J_{0})=\frac{1}{2}\sum_{i=1}^{2n}|J_{0}e_{i}|^{2}=n$.
And $$e(J_{t})=\frac{1}{2}\sum_{i}|J_{t}e_{i}|^{2}
=\frac{1}{2}\sum_{i,j}(J^{j}_{i})^{2}=\frac{1}{4}\sum_{i,j}(J^{j}_{i})^{2}+(J^{i}_{j})^{2}=\frac{1}{4}\sum_{i}\sum_{j}(J^{j}_{i})^{2}+(J^{i}_{j})^{2}$$
$$\geq \frac{1}{2}\sum_{i}|\sum_{j}J^{j}_{i}J^{i}_{j}|=\frac{1}{2}\sum_{i} 1=n=e(J_{0}),$$
where $J_{t}e_{i}=J_{i}^{j}e_{j}$. So the energy $E(J_{0})=\int_{M}e(J_{0})dv$ is minimal. Let us compare with harmonic maps. Let $f$ be a smooth map between two Riemannian
manifolds $M$ and $N$. If the energy $E(f)=\int_{M}e(f)dv$ is minimal, then $f$ is a harmonic map and the trace $\delta (df)=0$ \cite{[X]}. So intuitively we should has $\delta J=0$.

Combining theorem \ref{t3.3} and proposition \ref{p5.1}, we immediately have

\begin{theorem}\label{t5.2}
Let $M$ be an almost-Hermitian manifold. Then $dJ=0$ if and only if
\begin{equation} \int_{M}(|\nabla J|^{2}+S-\langle R(e_{i}, e_{j})Je_{i}, Je_{j}\rangle)dv=0, \label{f5.1} \end{equation}
where $S$ denotes the scale curvature.
\end{theorem}

We use proposition \ref{p5.1} to give a vanishing result related to Nijenhuis tensor. This may be known elsewhere.
\begin{theorem}
Let $M$ be an almost-Hermitian manifold. Then for any $X\in \Gamma(TM)$ we have $\langle N(J)(X,e_{i}),e_{i}\rangle\equiv0.$
\end{theorem}

\begin{proof}
By proposition \ref{p5.1}, $$\langle N(J)(X,e_{i}),e_{i}\rangle=\langle JN(J)(X,e_{i}),Je_{i}\rangle=\langle dJ(JX,Je_{i})-dJ(X,e_{i}), Je_{i}\rangle$$
 $$=\langle dJ(JX,Je_{i}), Je_{i}\rangle-\langle dJ(X,e_{i}), Je_{i}\rangle=-\langle JX, \delta J\rangle+\langle JX, \delta J\rangle=0.$$
\end{proof}

\subsection{Connections with the balanced metrics}
Let $M$ be an almost-Hermitian manifold. Let $\omega$ given by $\omega(X, Y)=\langle X, JY\rangle$ be the almost-Hermitian form. When $J$ is integrable, we say that $J$ induces a K\"{a}hler structure if $d\omega=0$ and a balanced structure if $d\omega^{m-1}=0$ ($m=\frac{\dim M}{2}$) (c.f. \cite{[M]}). We will show that

\begin{theorem}
Let $M$ be a compact almost-Hermitian $n$-manifold ($n=2m$). If formula \ref{f5.1} holds, then $J$ induces a balanced structure.
\end{theorem}

For a special case $n=4$, formula \ref{f5.1} implies that $J$ induces a K\"{a}hler structure.

Since formula \ref{f5.1} implies that $dJ=0$ (hence $J$ is integrable) and $\delta J=0$. Then theorem 5.4 follows from below lemma.
\begin{lemma}
$d\omega^{m-1}=0$ if and only if $\delta J=0$.
\end{lemma}

\begin{proof}
We use same notation $\delta$ denote the co-differential operator in Hodge theory. Then $d\omega^{m-1}=0$ if and only if $\delta\omega=0$. For any $X \in \Gamma(TM)$,
\begin{eqnarray*}
-\delta\omega(X)&=&i(e_{j})\nabla_{e_{j}}\omega (X)=(\nabla_{e_{j}}\omega) (e_{j},X)\\
&=&\nabla_{e_{j}}\omega (e_{j},X)-\omega (\nabla_{e_{j}}e_{j},X)-\omega (e_{j},\nabla_{e_{j}}X)\\
&=&e_{j}\langle e_{j},JX\rangle-\langle\nabla_{e_{j}}e_{j},JX\rangle+\langle Je_{j},\nabla_{e_{j}}X\rangle\\
&=& div(JX)+e_{j}\langle Je_{j},X\rangle-\langle\nabla_{e_{j}}Je_{j},X\rangle\\
&=& div(JX)-e_{j}\langle e_{j},JX\rangle+\langle \delta J,X\rangle\\
&=&\langle \delta J,X\rangle.
\end{eqnarray*}
\end{proof}

For a K\"{a}hler structure we have the well-known relation between curvature tensor and complex structure: $R(X,Y)JZ=JR(X,Y)Z$. For a balanced manifold, we have that
\begin{theorem}
 Let $M$ be a balanced manifold. Then $|S-\langle R(e_{i}, e_{j})Je_{i}, Je_{j}\rangle|\leq |\nabla J|^{2}.$
\end{theorem}

\begin{proof}
We choose a normal frame $\{e_{i}\}$. Then $|\nabla J|^{2}=\sum\limits_{i,j} |\nabla_{e_{i}}Je_{j}|^{2}$, $|dJ|^{2}=\sum\limits_{i<j}|\nabla_{e_{i}}Je_{j}-\nabla_{e_{j}}Je_{i}|^{2}$.
Then
\begin{eqnarray*}
|dJ|^{2}&=& \sum\limits_{i<j}(|\nabla_{e_{i}}Je_{j}|^{2}+|\nabla_{e_{j}}Je_{i}|^{2}-2\langle\nabla_{e_{i}}Je_{j},\nabla_{e_{j}}Je_{i}\rangle)\\
&\leq& 2\sum\limits_{i<j}(|\nabla_{e_{i}}Je_{j}|^{2}+|\nabla_{e_{j}}Je_{i}|^{2})\\
&=& 2(\sum\limits_{i,j}|\nabla_{e_{i}}Je_{j}|^{2}-\sum\limits_{k}|\nabla_{e_{k}}Je_{k}|^{2})\\
&\leq& 2(|\nabla J|^{2}-\frac{|\delta J|^{2}}{n}).
\end{eqnarray*}

When $\delta J=0$, by the formula \ref{f3.1} we can calculate directly that $ |dJ|^{2}=|\nabla J|^{2}+S-\langle R(e_{i}, e_{j})Je_{i}, Je_{j}\rangle$.
So one has $-|\nabla J|^{2}\leq S-\langle R(e_{i}, e_{j})Je_{i}, Je_{j}\rangle \leq|\nabla J|^{2}$.
\end{proof}

\begin{remark}
 Let $M$ be an almost complex $n$-manifold. The $J$ needs not be compatible with the Riemannian metric.  Then $$-\int_{M}|\nabla J|^{2}dv \leq \int_{M}(\langle JR(e_{i}, e_{j})e_{i},Je_{j}\rangle-\langle R(e_{i}, e_{j})Je_{i}, Je_{j}\rangle)dv\leq (n-1)\int_{M}|\nabla J|^{2}dv.$$
\end{remark}

We leave the proof of remark to the readers.

\section{Some remarks}
Our main purpose is to find some geometric obstructions or sufficient conditions to existence of complex structures. The obstructive problem is try to find a
necessary and sufficient condition for Nijenhuis tensor vanishing. Namely $$N(J)=0 \Longleftrightarrow C(J)=0.$$ Here $C(J)$ is a global curvature
expression related to $J$, which is similar to formula \ref{f4.1} or \ref{f5.1}. Once we find such a relation, if a suitable geometric condition
can leads to $C(J)>0(<0)$ for any $J$, we can claim that the manifold does not admit a complex structure. Theorem \ref{t4.1} or \ref{t5.2} only gives a sufficient
 condition for Nijenhuis tensor vanishing.

To find an effective sufficient condition, professor Kefeng Liu suggests that we should use \ref{f3.1} to construct a suitable flow of almost
 complex structures. Under some suitable geometric condition the flow converges to an integrable one. In this case we must deal with the problem
  of how the evolution equation keeps the almost complex structures.

\bibliographystyle{amsplain}

\end{document}